\documentclass[11pt]{amsart}
\usepackage{amssymb,amsmath,amsthm,fullpage}
\usepackage{pdfsync,enumerate,color}
\usepackage{graphicx}
\usepackage[colorlinks]{hyperref}

\newcommand{\eps}{\varepsilon}

\newcommand{\opF}{\mathcal{F}}

\def\m{\medskip}

\newcommand{\opH}{\mathcal{H}}

\newcommand{\R}{\mathbb R}

\newcommand{\C}{\mathbb C}
\DeclareMathOperator{\supp}{supp}
\DeclareMathOperator{\Imm}{Im}

\DeclareMathOperator{\Rre}{Re}

\DeclareMathOperator{\sgn}{sgn}

\newcommand{\z}{\bar z}

\newcommand{\dbarb}{\bar\partial_b}
\newcommand{\dbarbs}{\bar\partial^*_b}

\newcommand{\Boxb}{\Box_b}

\newcommand{\nn}{\nonumber}
\newcommand{\ep}{\epsilon}
\newcommand{\I}{\mathcal{I}}

\newtheorem{thm}{Theorem}[section]
\newtheorem{prop}[thm]{Proposition}

\newtheorem{cor}[thm]{Corollary}

\newtheorem*{theorem*}{Theorem}

\theoremstyle{definition}
\newtheorem{defn}[thm]{Definition}

\theoremstyle{remark}
\newtheorem{rem}[thm]{Remark}

\begin{document}

\title[Fundamental Solution to $\Boxb$]{The Fundamental Solution to $\Box_b$ on Quadric Manifolds -- Part 1. General Formulas}%

\author{Albert Boggess and Andrew Raich}%
\address{
School of Mathematical and Statistical Sciences\\
Arizona State University\\
Physical Sciences Building A-Wing Rm. 216\\
901 S. Palm Walk\\
Tempe, AZ 85287-1804 }
\address{
Department of Mathematical Sciences \\ 1 University of Arkansas \\ SCEN 327 \\ Fayetteville, AR 72701}

\email{boggess@asu.edu, araich@uark.edu}

\thanks{This work was supported by a grant from the Simons Foundation (707123, ASR)}%

\begin{abstract} This paper is the first of a three part series in which we explore geometric and analytic properties of the Kohn Laplacian and its inverse on general quadric submanifolds of $\mathbb{C}^n\times\mathbb{C}^m$. In this paper, we present a streamlined calculation for a general integral formula for the complex Green operator $N$ and the projection onto the nullspace of $\Box_b$. The main application of our formulas is the critical case of codimension two quadrics in $\mathbb{C}^4$ where we discuss the known solvability and hypoellipticity criteria of Peloso and Ricci \cite{PeRi03}. We also provide examples to show that our formulas yield explicit calculations in some well-known cases: the Heisenberg group and a Cartesian product of Heisenberg groups.
\end{abstract}

\maketitle

%
%
\section{Introduction}\label{sec:intro}

The goal of this paper is to present an explicit integral formula for the complex Green operator and the projection onto the null space of
the Kohn Laplacian
on quadric submanifolds of $\C^n\times\C^m$. 
Our result generalizes the formula of \cite{BoRa13q} from the specific case of codimension $2$ quadrics in $\C^4$ 
to the general case of arbitrary $n$ and $m$, and we also prove a formula for the complex Green operator when
it is only a relative inverse of the Kohn Laplacian and not a full inverse. 
Additionally, the new proof is significantly simpler and uses
our calculation of the $\Boxb$-heat kernel on general quadrics \cite{BoRa11}. We then provide several applications of our
formula to the case of codimension $2$ quadrics in $\C^4$ and provide context for how this case fits into the solvability/hypoellipticity framework of
Peloso and Ricci \cite{PeRi03}.  We also provide a calculation of the complex Green operator in several instances when it is only a relative
fundamental solution: $\Boxb$ on the Heisenberg group on functions and $(0,n)$-forms, and on the Cartesian product of Heisenberg groups at the top degree. 
We conclude with a computation of the Szeg\"o projection on the Cartesian product of Heisenberg groups.

This paper is the first of a series where we explore the geometry and analysis of
the Kohn Laplacian $\Boxb$ and its (relative) inverse, the complex Green operator, on quadric submanifolds
in $\C^n\times\C^m$. 
The $\Boxb$-equation, $\Boxb u =f$, governs the behavior of boundary values of
holomorphic functions, and the $\Boxb$-operator is a naturally occurring, non-constant coefficient, non-elliptic operator. 
Solving the equation has been tantalizing mathematicians for the better part of fifty years, and while much is known for solvability/regularity in 
$L^p$ and other spaces (especially $L^2$) for hypersurface type CR manifolds, 
there has been much less work done to determine the structure of the complex Green operator, denoted by $N$, especially on
non-hypersurface type CR manifolds. 
The problem is that the techniques used to solve the equation are functional analytic in nature and therefore
non-constructive. Consequently, to have any hope of finding an explicit solution, we need additional structure on the CR manifold. 
From our perspective, the gold standard for results in this area is the calculation by Folland and Stein \cite{FoSt74p} in which
they find a beautiful, closed form expression for $N$ on the Heisenberg group. 

In the decades following \cite{FoSt74p}, mathematicians developed machinery to solve the $\Boxb$-heat equation 
(or the heat equation associated to the sub-Laplacian) on certain Lie groups. From the formulas for the heat equations, in principle, it is only a matter of
integrating the time variable out to recover the formula for $N$.
The first results for the heat equation were for the sub-Laplacian on the Heisenberg group by Hulanicki \cite{Hul76} and Gaveau \cite{Gav77}. 
More results followed in a similar vein, and the vast majority rely on the group Fourier
transform and Hermite functions \cite{PeRi03, BoRa09, BoRa11, YaZh08,CaChTi06,BeGaGr96,BeGaGr00, Eld09}. The problem with these 
techniques is that the formulas that they generate for the heat kernel 
are only given up to a partial Fourier transform that is uncomputable in practice. Consequently, any information
giving precise size estimates or asymptotics, let alone a formula in the spirit of Folland and Stein, is absent. 

A quadric submanifold $M\subset\C^n\times\C^m$ is a CR manifold of the form
\begin{equation}\label{eqn:quadric}
M= \{(z,w) \in \C^n \times \C^m: \ \Imm w = \phi(z,z) \}
\end{equation}
where $\phi: \C^n \times \C^n \mapsto \C^m $ is a sesquilinear
form (i.e., $\phi (z, z') = \overline{\phi (z',z)}$). 
The fundamental solution to $\Box_b$ or the sub-Laplacian on quadrics has been studied many authors, 
including \cite{BeGr88,BeGaGr96,BoRa09,BoRa13q,CaChTi06,FoSt74p,PeRi03}.  See also Part II of the series, \cite{BoRa20II}, in which 
we find useable sufficient conditions for a map $T$ between quadrics to be a $\Box_b$-preserving Lie group isomorphism as well as 
establish a framework for which appropriate derivatives of the complex Green operator are continuous in $L^p$ and $L^p$-Sobolev spaces (and hence are hypoelliptic). We apply the general results to codimension $2$ quadrics in $\mathbb{C}^4$.

There are two higher codimension papers that need mentioning. First, in \cite{NaRiSt01}, Nagel et. al. analyze $\Boxb$ and
its geometry in the special
class of decoupled quadrics where $\Imm w_j = \sum_{k=1}^n a_k^j|z_k|^2$. 
However, many of the interesting cases do not fall into this category.
Second, Raich and Tinker compute the 
the Szeg\"o kernel for the polynomial model
\[
M = \{(z,w)\in\C\times\C^n : \Imm w = (a_1,\dots,a_{n-1},1) p(\Rre z) \big\}
\]
where $p:\R\to\R$ is a smooth function satisfying $\lim_{|x|\to\infty}\frac{p(x)}{|x|} = \infty$, 
and the constants $a_1,\dots,a_{n-1}$ are nonzero \cite{RaTi15}. 
The authors write an explicit formula for the Szeg\"o kernel based on an integral formula of Nagel \cite{Nag86} and show that there are significant blowups off of the diagonal. Raich and Tinker evaluate all of the integrals in the case
$p(x) = x^2$. That is, the CR manifold $M$ is a quadric, but this case is very special because the tangent space (at each point) has only one complex direction, so every degree is either top or bottom. These are often the exceptional cases and can give
misleading intuition.

Associated to each quadric is the \emph{Levi form} $\phi(z,z')$ and for each $\lambda\in\R^m$, the \emph{directional Levi form in the direction
$\lambda$} is $\phi^\lambda(z,z') = \phi(z,z')\cdot\lambda$ where $\cdot$ is the usual dot product without conjugation.
For each $\lambda$, there is an $n\times n$ matrix $A^\lambda$ so that
\[
\phi^\lambda(z,z') = z^* A^\lambda z'
\]
and we identify the eigenvalues of $\phi^\lambda$ with the eigenvalues of the matrix $A^\lambda$.
Here, the $*$ designates the Hermitian transpose.

The outline of the remainder of the paper is as follows: In Section \ref{sec:results}, we record the main result for general quadrics
and provide additional context. In Section \ref{sec:preliminaries}, we discuss the CR geometry and Lie group structure of a general quadric.
In Section \ref{sec:new derivation}, we prove the main result, and we devote Section \ref{sec:examples} to explicit examples.

\subsection{Acknowledgements} The authors would like to express their sincere appreciation for the detailed review from the referee and the excellent editorial comments from Harold Boas. 

%
%
\section{Results and Discussion}\label{sec:results}
Under the projection $\pi:\C^n\times\C^m \to \C^n\times\R^m$ given by $\pi(z,t+is)=(z,t)$, we may identify a quadric $M$ with $\C^n\times\R^m$. 
The projection induces both a CR structure and Lie group structure on $\C^n\times\R^m$, and we denote this Lie group by $G$ (or $G_M$). 
Thus the projection is a CR isomorphism and we refer to the pushforwards and pullbacks of objects from $M$ to $G$ without changing the
notation.

\subsection{The Main Result -- General Formulas for the Solution of $\Box_b$ on Quadrics.} \label{subsec:genformula}

To state the main result, we need to introduce some notation.
For each $\lambda\in\R^m\setminus\{0\}$, let $\mu^\lambda_1,\dots,\mu^\lambda_n$ be the eigenvalues of $\phi^\lambda$ (or equivalently, the Hermitian symmetric matrix, $A^\lambda$)
and $v^\lambda_1, \dots, v^\lambda_n$ be an orthonormal set of eigenvectors. This means
\begin{equation}
\label{eq: ortho}
\phi^\lambda (v^\lambda_j, v^\lambda_k) = \delta_{jk} \mu_j^\lambda .
\end{equation}
Let $\alpha = \lambda/|\lambda|$, then $\mu_j^\lambda = |\lambda|\mu^\alpha_j$ and $v_j^\lambda = v_j^\alpha$.
If $z\in\C^n$ is expressed in terms of the
unit eigenvectors of $\phi^\lambda$, then $z_j^\lambda = z_j^\alpha$ is given by
\[
z^\alpha = Z(z, \alpha) = ({U^\alpha })^*\cdot z
\]
where $U^\alpha$ is the matrix whose columns are the eigenvectors, $v^\alpha_j$, $1 \leq j \leq n$,  and $\cdot $ represents matrix multiplication with $z$ written as a column vector. Note that the corresponding orthonormal basis of $(0,1)$-covectors for this basis
is 
\[
d \bar Z_j(z, \alpha), \ \ 1 \leq j \leq n, \ \ \textrm{where} \ \ 
d \bar Z(z, \alpha) =  (U^\alpha)^T \cdot d \bar z
\]
where $d \bar z$ is written as a column vector of $(0,1)$-forms and the superscript $T$ stands
for transpose. 
It is a fact that the eigenvalues, eigenvectors and hence $z^\alpha =  Z(z, \alpha) $ depend smoothly on $z \in \C^n$.
However, while the dependence of the eigenvalues is continuous (in fact Lipschitz) in $\alpha$, the eigenvectors may only be functions of bounded variation (SBV) in $\alpha$ \cite[Theorem 9.6]{Rai11}.

Let $\I_q = \{ L = (\ell_1,\dots,\ell_q): 1 \leq \ell_1 < \ell_2 <\cdots < \ell_q \leq n\}$. 
For each $K \in \I_q$, we will need to express $d \bar z^K$, in terms of $d \bar Z(z, \alpha)^L$ for $L \in \I_q$. We have
\begin{equation}
\label{C(K,L)}
d \bar z^K = \sum_{L\in\I_q} C_{K,L}(\alpha) \, d \bar Z(z, \alpha)^L
\end{equation}
where $C_{K,L} (\alpha)$ are the appropriate $q \times q$ minor determinants of $\bar U^\alpha$. 
Note that if $q =n$, then the above sum only has one term and $C_{K,K} (\alpha)= 1$. Additionally, when $q=0$, $\I_0 = \emptyset$ and the sum \eqref{C(K,L)} does not appear.

Denote by
$\nu(\lambda)=\nu(\alpha)$ the number of nonzero eigenvalues of $\phi^\lambda$. For each $q$-tuple $L \in \I_q$, set
\[
\Gamma_L = \{\alpha\in S^{m-1}: \mu^\alpha_\ell >0 \text{ for all }\ell\in L \text{ and } \mu^\alpha_\ell < 0 \text{ for all }\ell\not\in L\}
\] 
and
\[
\eps_{j,L}^\alpha = \begin{cases}\sgn(\mu_j^\lambda) & \text{if }j \in L \\
- \sgn(\mu_j^\lambda) &\text{if }j\not\in L. \end{cases}
\] 
\begin{rem} 
If $\nu = \max_{\lambda\in\R^m} \nu(\lambda)$, then $\{\lambda \in\R^m\setminus\{0\} : \nu(\lambda) = \nu\}$
is a Zariski open set and hence carries full Lebesgue measure. In particular, if
 one of the sets $\Gamma_L$ is nonempty, then $\nu=n$. When $\nu < n$, we arrange our eigenvalues so that
$\mu^\lambda_{\nu+1} = \cdots \mu^\lambda_n=0$ and write $z'=(z_1,\dots,z_\nu)$ and $z'' = (z_{\nu+1},\dots,z_n)$.
\end{rem}

\begin{defn}
Given an index in $K \in \I_q$, we say that a current 
$N_K = \sum_{K'\in\I_q} \tilde N_{K'}(z,t)\, d\z^{K'}$ is a {\em fundamental solution to $\Box_b$} 
on forms spanned by $d \bar z^K$ if $\Box_b N_K = \delta_0(z,t) \, d\bar z^K$. 
\end{defn}
$N_K$ acts on smooth forms with compact support by componentwise convolution with respect to the group structure on $G$, that is, if $f = f_0\, d\z^K$, then
$N_K*f = \sum_{K'\in\I_q}  \tilde N_{K'} * f_0\, d\z^{K'}$.
Thus if $f =f_0 d \bar z^K$ is a smooth form with compact support, then $\Box_b \{N_K * f \} = f$. In cases where $\Box_b$ has a nontrivial kernel, we let $S_K $ be the projection (Szeg\"o) operator onto this kernel and we say that $N_K$ is a {\em relative fundamental solution} if
$ \Box_b \{N_K * f \} = f-S_K(f)$ holds for all compactly supported forms spanned by $d \bar z^K$.
On quadrics, $\Boxb$ never has closed range in $L^2$ so the complex Green operator cannot be continuous in $L^2$.
As a consequence, we can only discuss \emph{a} relative inverse and not \emph{the} relative inverse. However, a relative inverse is called 
\emph{canonical} if its output is orthogonal to $\ker\Boxb$
whenever it belongs to $L^2$.

We can now state our main result.

\begin{thm}\label{thm:new derivation}
Suppose $M$ is a quadric CR submanifold of $\C^{n+m}$ given by (\ref{eqn:quadric})
with associated projection $G$. Fix $K\in\I_q$.
\begin{enumerate}[1.]
\item If $|\Gamma_L|=0$ for  all $L \in \I_q$, then the fundamental solution to $\Box_b$ on forms spanned by $d\z^K$ is given by 
\begin{align}\label{eqn:N no Szego}
N_K (z,t) &=\frac{4^n}{2(2 \pi)^{m+n}}  \sum_{L \in \I_q} 
\int_{\alpha \in S^{m-1}} C_{K,L} (\alpha) \, d \bar Z(z, \alpha)^L   \int_{r=0}^1  
\frac{1}{|\log r|^{n-\nu(\alpha)}} 
\prod_{j=1}^{\nu(\alpha)} \frac{r^{\frac 12(1-\eps_{j,L}^\alpha) |\mu_j^\alpha |}|\mu_j^\alpha |}{1- r^{|\mu_j^\alpha|}}\nn \\
& \ \ \ \  \frac{ (n+m-2)! }{  (A_\alpha(r,z)-i \alpha \cdot t)^{n+m-1}}  \frac{dr \, d \alpha}{r}
\end{align}
where
\[
A_\alpha(r,z) = \frac{2}{|\log r|}|{z''}^\alpha|^2 +  \sum_{j=1}^{\nu(\alpha)} |\mu_j^\alpha| \left( 
\frac{1+ r^{|\mu_j^\alpha|}}{1- r^{|\mu_j^\alpha|}} \right) |z_j^\alpha|^2 .
\]
\item If $|\Gamma_L|>0$ for at least one $L \in \I_q$,  then orthogonal projection onto the $\ker\Boxb$ applied to forms
spanned by $d\z^K$ is given by convolution with the 
$(0,q)$-form:
\begin{equation}\label{eqn:Szego}
S_K (z,t) = \frac{4^n(n+m-1)!}{(2\pi)^{m+n}} 
\sum_{L \in \I_q} 
\int_{\alpha \in \Gamma_L} C_{K,L} (\alpha) \, d \bar Z(z, \alpha)^L
\frac{\prod_{j=1}^n|\mu_j^\alpha|}{(\sum_{j=1}^n |\mu_j^\alpha||z_j^\alpha|^2 - i \alpha\cdot t)^{n+m}}\, d\alpha. 
\end{equation}
In the case $K = \emptyset$, $S_\emptyset(z,t)$ is the Szeg\"o kernel.
\item If $|\Gamma_L|>0$ for at least one $L \in \I_q$, then the canonical relative fundamental solution 
to $\Box_b$ given by $\int_0^\infty e^{-s\Boxb}(I-S_q) \, ds$  applied to forms spanned by $d\z^K$ is given by 
\begin{align}
&N_K (z,t) =  \nn \\
&\frac{4^n(n+m-2)!}{2(2 \pi)^{m+n}}  
\sum_{L \in \I_q} \left(
\int_{\alpha \not\in\Gamma_L} C_{K,L} (\alpha) \, d \bar Z(z, \alpha)^L 
\int_{r=0}^1  
\prod_{j=1}^n \frac{r^{\frac 12(1-\eps_{j,L}^\alpha) |\mu_j^\alpha |}|\mu_j^\alpha |}{1- r^{|\mu_j^\alpha|}}
\frac{ 1}{  (A(r,z)-i \alpha \cdot t)^{n+m-1}}  \frac{dr \, d \alpha }{r} \right. \nn \\
&+ 
\int_{\alpha \in\Gamma_L} C_{K,L} (\alpha) \, d \bar Z(z, \alpha)^L \nn \\
& \left. \ \ \ \int_{r=0}^1 
\Big[\Big(\prod_{j=1}^n \frac{|\mu_j^\alpha|}{1-r^{|\mu_j^\alpha|}}\Big) \frac{1}{(A_\alpha(r,z)-i\alpha\cdot t)^{n+m-1}}  - \frac{\prod_{j=1}^n|\mu_j^\alpha|}{(A_\alpha(0,z)-i\alpha\cdot t)^{n+m-1}}\Big] \frac{dr \, d \alpha}r \right)  \nn \\
\label{eqn:N with Szego}
\end{align}
where
\[
A_\alpha(r,z) =  \sum_{j=1}^{n} |\mu_j^\alpha| \left( 
\frac{1+ r^{|\mu_j^\alpha|}}{1- r^{|\mu_j^\alpha|}} \right) |z_j^\alpha|^2 .
\]
\end{enumerate}
In all cases, the integrals converge absolutely.
\end{thm}

\begin{rem} \label{rem:C_KL diagonal}
In many of the most important cases, the functions $C_{K,L}(\alpha) = \delta_{KL}$, and formulas from the theorem simplify.
There are several cases when this simplifcation occurs.  
The first is when $q=0$ or $q=n$. The second is when the orthonormal basis $\{ v^\lambda_j \}$ is independent of $\lambda$. 
This independence happens both when $m=1$ (the hypersurface type case)
or in the sum of squares case considered by Nagel, Ricci, and Stein \cite{NaRiSt01}, discussed in Section \ref{sec:intro}.
\end{rem}

\begin{rem}
It is a straightforward exercise to recover the classical complex Green operator on the 
Heisenberg group from (\ref{eqn:N no Szego}) \cite{BoRa13q}. Additionally, \cite[Theorem 2]{BoRa13q} is now a simple and immediate
application of (\ref{eqn:N no Szego}).

\end{rem}

\subsection{Solvability, hypoellipticity, and $\phi^\lambda$} In \cite{PeRi03}, Peloso and Ricci 
say that
\begin{enumerate}[1.]
\item $\Boxb$ is \emph{solvable} if given any smooth $(0,q)$-form $\psi$ on $G$ with compact support, there exists a $(0,q)$-current $u$ on $G$
so that $\Boxb u = \psi$;
\item $\Boxb$ is \emph{hypoelliptic} if given any $(0,q)$-current $\psi$ on $G$, $\psi$ is smooth on any open set on which $\Boxb\psi$ is smooth.
\end{enumerate}
Peloso and Ricci are able to characterize solvability and hypoellipticity of $\Boxb$.
\begin{thm}[\cite{PeRi03}] \label{thm:P-R}
Let $n^+(\lambda)$, resp., $n^-(\lambda)$, be the number of positive, resp., negative eigenvalues of $\phi^\lambda$.
Then 
\begin{enumerate}[1.]
\item $\Boxb$ is solvable on $(0,q)$-forms if and only if there does not exist $\lambda\in\R^m\setminus\{0\}$ for which $n^+(\lambda) = q$ and
$n^-(\lambda) = n-q$.
\item $\Boxb$ is hypoelliptic on $(0,q)$-forms if and only if there does not exist $\lambda\in\R^m\setminus\{0\}$ for which $n^+(\lambda) \leq q$ and
$n^-(\lambda) \leq n-q$.
\end{enumerate}
\end{thm}

\begin{rem}
The condition $|\Gamma_L|>0$ is equivalent to the nontriviality of $S_K(z,t)$ and
is easy to check. By combining the solvability criteria of \cite{PeRi03} (Theorem \ref{thm:P-R}) 
and the formula for the $\Boxb$-heat kernel from \cite{BoRa11} (Theorem \ref{heatkerthm} below), it must be the case that
$\nu = n$ (see (\ref{heatker})) as $S_L(z,t) = \lim_{s\to\infty} H_L(s,z,t)$ and solvability is equivalent $S_L(z,t)=0$. The latter statement follows from the
fact the condition in part 1 of Theorem \ref{thm:P-R} is an open condition, that is, when solvability fails,
$\Gamma_L$ is a (union of) cones, at least one of which will be open and hence has nonzero measure.
\end{rem}

%
%
\section{The Kohn Laplacian on Quadrics}\label{sec:preliminaries}

For a discussion of the group theoretic properties of $G$, please see \cite{BoRa11} or \cite{PeRi03}. By definition, the operator $\Box_b$ is defined on $(0,q)$ forms as
$\Box_b = \bar \partial_b \bar \partial_b^* + \bar \partial_b^* \bar \partial_b$ without reference to any particular coordinate system. However in order to do computations, we need formulas for $\Box_b$ with respect to carefully chosen coordinates.

For $v \in \R^{2n} \approx \C^n$, let $\partial_v$ be the real
vector field given by the directional 
derivative in the direction of $v$. Then
the right invariant vector field at an arbitrary
$g=(z,w) \in M$ corresponding to $v$ is given by
\[
X_v (g)=  \partial_v + 2 \Imm \phi (v,z) \cdot D_t
=\partial_v - 2 \Imm \phi (z,v) \cdot D_t.
\]
Let $Jv$ be the vector in $\R^{2n}$ which corresponds
to $iv$ in $\C^n$ (where $i=\sqrt{-1}$).
The CR structure on $G$ is then spanned by vectors of the 
form:
\begin{equation}
\label{Z-formula}
Z_v(g)=(1/2)(X_v-iX_{Jv})
=(1/2)(\partial_v - i \partial_{Jv}) -i \overline{\phi (z,v)} \cdot D_t 
\end{equation}
and
\begin{equation}
\label{Zbar-formula}
\bar Z_v(g)=(1/2)(X_v+iX_{Jv})
=(1/2)(\partial_v + i \partial_{Jv}) +i \phi (z,v) \cdot D_t. 
\end{equation}

Let $v_1, \dots  , v_n$ be any orthonormal basis for $\C^n$.
Let $X_j= X_{v_j}$, $Y_j=X_{Jv_j}$,
and let $Z_j=(1/2) (X_j-iY_j)$, $\bar{Z}_j = (1/2) (X_j+iY_j)$
be the right invariant CR vector fields defined above 
(which are also the left invariant vector fields for the group structure
with $\phi$ replaced by $-\phi$).
A $(0,q)$-form can be expressed as $\sum_{K\in \I_q} \phi_K\, d\z^K $.
An explicit  formula for $\Box_b$ on quadrics is written down by Peloso and Ricci 
\cite{PeRi03} (see also \cite{BoRa11}) which takes the following form: if $\phi = \sum_K \phi_K\, d \bar z^K$ is a $(0,q)$-form, then 
\begin{equation}
\label{boxb}
\Box_b ( \sum_{K\in\I_q} \phi_K\, d\z^K )
= \sum_{K,L\in \I_q}  \Box_{LK}^v \phi_K \, d \bar{z}^L
\end{equation}
where 
\[
\Box_{LL}^v = (1/2) \sum_{\ell =1}^n\big( Z_\ell \bar Z_\ell + \bar Z_\ell  Z_\ell\big)
+(1/2) \left( \sum_{\ell \in L}  [  Z_\ell,  \bar Z_\ell]- \sum_{k \not \in L} [Z_k, \bar Z_k] \right)
\]
If $L \not=K$, then $\Box_{LK}^v$ is zero unless $|L \cap K| = q-1$, in which case 
\[
\Box_{LK}^v=(-1)^{d_{kl}} [Z_k, \bar Z_\ell]
\]
where $k \in K$ is the unique element not in $L$ and $\ell \in L$ is the unique element not in $K$
and $d_{kl}$ is the number of indices between $k$ and $\ell$. The notation $\Box_{LK}^v$ indicates the dependency of this differential operator on the particular orthonormal basis
$v = (v_1, \dots, v_n)$ chosen and the resulting basis (i.e., $Z_1, \dots, Z_n$) and the associated dual basis of $(0,1)$-forms (i.e., $d \bar z_1, \dots , d \bar z_n$).

Note that if $|L \cap K| = q-1$, then $\Box_{LK}^v$ is quite simple since it is a linear combination over $\C$ of $t_j$ derivatives. In the next section, we will use the coordinates 
$z^\alpha = Z(z, \alpha)$ derived from the basis $v^\alpha_1, \dots, v^\alpha_n$ used in Section \ref{sec:results}, and we will see that we can ignore the 
$\Box_{LK}^{v^\alpha}$ when $L \not= K$.

\subsection{\textbf{Fourier Transform of $\Box_b$}}
Since the quadric defining equations are independent of $t \in \R^m$, we can use the  Fourier transform in the $t$-variables:
\[
\hat f( \lambda) = \frac{1}{(2\pi)^{\frac m2}} \int_{\R^m} f(t) e^{-i \lambda\cdot t}  \, dt.
\]
In the case that $f$ is a function of $(z,t)$, we use the notation $f(z,\hat\lambda)$ to denote the partial Fourier transform of $f$
in the $t$-variables.
We transform $\Box_b$ via the Fourier transform and consider the fundamental solution to the heat operator in the transformed variables. 
We then use the $z^\alpha$
coordinates relative to the basis $v_j^\alpha$ chosen in Section \ref{sec:results}  
for the $z$-variable in $f(z^\alpha, \hat \lambda)$ with $\alpha = \frac{\lambda}{|\lambda|}$. 
Thus, $\lambda$ plays two roles - first as the Fourier transform variable and second, as the label for the coordinates relative to the
basis $v_j^\alpha$ which diagonalizes $\phi^\lambda$.  Also note that the operation of Fourier transform in $t$ and the operation of expressing $z$ in terms of the $z^\alpha$ coordinates are 
interchangeable (i.e., these operations commute).

For a general orthonormal basis $v =\{v_1, \dots, v_n\}$, let $\Box_{LL}^{v,\hat \lambda}$ be the partial Fourier transform in $t$ of the sub-Laplacian $\Box_{LL}^v$. When $v=v^\alpha$, we have (from \cite{BoRa11}):
\[
\Box_{LL}^{v^\alpha, \hat \lambda} = -\frac{1}{4} \Delta +2i \sum_{k=1}^n \mu_k^\lambda \Imm \{ z_k^\alpha \partial_{z_k^\alpha} \}
+ \sum_{k=1}^n (\mu_k^\lambda)^2 |z_k^\alpha|^2 - \left( \sum_{k \in L} \mu_k^\lambda - \sum_{k \not\in L}
\mu_k^\lambda \right)
\]
where $\Delta$ is the ordinary Laplacian in $z=z^\alpha $ coordinates. 

Also note that  
\begin{align*}
\Box^{v^\alpha}_{LK} &= (-1)^{d_{k \ell}} [Z_{v_k^\alpha} , \bar Z_{v_\ell^\alpha}] \\
&=2i (-1)^{d_{k \ell}} \textrm{Re} \, \phi(v_k^\alpha, v_\ell^\alpha) \cdot D_t
\ \ \ \textrm{using (\ref{Z-formula}) and (\ref{Zbar-formula})} .
\end{align*}
Using (\ref{eq: ortho}), we conclude that the Fourier transform of $\Box^{v^\alpha}_{LK}$ is
\[
\Box^{v^\alpha, \hat \lambda}_{LK} = -2 (-1)^{d_{k \ell}} \textrm{Re} \, \phi(v_k^\alpha, v_\ell^\alpha) \cdot \lambda =0
\]
when $L \not=K$ (i.e., when $\ell \not=k$). The signficance of this calculation is that 
the partial Fourier transform of $\Box_{LK}$ (expressed in global coordinates) is incorporated into the operators $\Box^{v^\alpha,\hat\lambda}_{LL}$.
 
\m

Next, we recall the heat kernel and Szeg\"o kernel for the $\Box_{LL}^{v^\alpha, \hat \lambda} $ heat equation. Let
$\tilde H_L (s, z^\alpha ,\hat \lambda)$ be the ``heat kernel'', i.e., the solution to the following boundary value problem:
\begin{align*}
\left[ \frac{\partial}{\partial s} + \Box_{LL}^{v^\alpha,  \hat \lambda} \right] \{\tilde H_L(s, z^\alpha,\hat \lambda) \} &= 0 \ \ \textrm{for} \ s>0 \\
\tilde H_L(s=0, z^\alpha, \hat \lambda) &=(2 \pi)^{-m/2} \delta_0 (z^\alpha) \\
&= (2 \pi)^{-m/2} \delta_0 (z) 
\end{align*}
where $\delta_0$ is the Dirac-delta function centered at the origin in the $z$ variables. Let 
$\tilde S_L (z^\alpha, \hat \lambda)$ be the Szeg\"o kernel which represents orthogonal projection of $L^2 (\C^n)$ onto the kernel of $\Box_{LL}^{v^\alpha, \hat \lambda}$.  Note, the tilde over the $\tilde H_L$ and $\tilde S_L$ indicates that these terms are functions rather than differential forms. By contrast, $N_K$ and $S_K$ in Theorem \ref{thm:new derivation} do not have tildes and they are differential $(0,q)$-forms.

\begin{thm} \label{heatkerthm} \cite{BoRa11} Let $L \in \I_q$ be a given multiindex of length $q$ and fix a nonzero $\lambda \in \R^m$
and $\alpha = \frac{\lambda}{|\lambda|}$. Then

\begin{enumerate}

\item The heat kernel which solves the above boundary value problem is
\begin{equation}
\label{heatker}
\tilde H_L (s, z^\alpha, \hat \lambda) =\frac{2^{n-\nu(\alpha)}}{(2 \pi)^{m/2+ n}s^{n-\nu(\alpha)}} e^{-\frac{|z''\{\alpha\}|^2}s}
\prod_{j=1}^{\nu(\alpha)} \frac{2e^{s \eps_{j,L}^\alpha |\mu_j^\lambda| }| \mu_j^\lambda| }
{\sinh (s |\mu_j^\lambda|  )}
e^{- |\mu_j^\lambda|   \coth (s |\mu_j^\lambda|  ) |z_j^\alpha|^2}
\end{equation}

\item If $\alpha\in\Gamma_L$, then 
the projection onto $\ker \Box_{LL}^{v^\alpha, \hat \lambda}$ is given by 
\begin{equation}
\label{SL}
 \tilde S_L (z^\alpha, \hat \lambda)=\lim_{s \to \infty}   \tilde H_L(s, z^\alpha,\hat \lambda) = 
\frac{4^n}{(2\pi)^{n+m/2}} \prod_{j=1}^n |\mu_j^\lambda|
e^{- |\mu_j^\lambda| |z_j^\alpha|^2},
\end{equation}
otherwise 
$\tilde S_L (z^\alpha, \hat \lambda)=0$. 
 
\item \label{fund-sol-lambda} The connection between the fundamental solution to the heat equation and the canonical relative fundamental solution to 
$\Box_{LL}^{v^\alpha, \hat \lambda}$, denoted $\tilde N_L(z^\alpha, \hat \lambda)$, is given as follows:
\begin{equation}
\label{Nfromheat}
 \tilde N_L(z^\alpha,\hat \lambda) = \int_0^\infty  \left[  \tilde H_L(s, z^\alpha,\hat \lambda)- \tilde S_L (z^\alpha, \hat \lambda) \right] \, ds.
\end{equation}
In particular, 
\begin{equation}
\label{NLsolve}
\Box_{LL}^{v^\alpha, \hat \lambda} \{  \tilde N_L(z^\alpha,\hat \lambda) \}
= (2 \pi )^{-m/2} (\delta_0(z) - \tilde S_L(z^\alpha, \hat \lambda) )
\end{equation}

 \end{enumerate}
 
 \end{thm}
 
 Both of the kernels $\tilde H_L(s,\cdot,\hat \lambda)$ and $\tilde S_L(\cdot,\hat\lambda)$ act on $L^2 (\C^n)$ via a twisted convolution, $*_\lambda$, where 
 $(f*_\lambda g)(z) = \int_{w \in \C^n} f(w)g(z-w) e^{-2i \lambda \cdot \Imm \phi (z,w)} \, dw$, 
 as defined in Section 5.4 of 
 \cite{BoRa11}, but this plays no role here.

\m

Let $\opF_\lambda^{-1}$ denote the inverse Fourier transform in $\lambda$ - that is, if 
$\tilde f(z, \lambda)$ is an integrable function of $\lambda \in \R^m$, then
\[
\opF_\lambda^{-1} (\tilde f(z, \lambda))(t) :=\frac{1}{2^{m/2}} \int_{\lambda \in \R^m}
\tilde f(z, \lambda) e^{i \lambda \cdot t} \, d \lambda.
\]
Now we can formulate our relative solution to $\Box_b$ and Szeg\"o kernel in terms of the inverse Fourier transform. 

\begin{prop}\label{prop:troublesome proposition}
For a given index $K \in \I_q$, the relative fundamental solution to $\Box_b$ applied to a form spanned by $d\z^K$
given  by $\int_0^\infty e^{-s\Boxb}(I-S_K)\, ds$ is
\begin{equation}
\label{rel-sol}
N_K (z,t) = \opF_\lambda^{-1} \left\{  \sum_{L \in \I_q} C_{K,L} (\alpha) \tilde N_L (z^\alpha, \hat \lambda) d \bar Z(z, \alpha)^L\right\}(t)
\end{equation}
Moreover, the orthogonal projection onto the $\ker\Boxb$ applied to forms
spanned by $d\z^K$ is given by convolution with the $(0,q)$-form
\begin{equation}
\label{szego}
S_K(z,t) =
\opF_\lambda^{-1} \left\{ (2 \pi)^{-m/2}\sum_{L \in \I_q} C_{K,L} (\alpha) \tilde S_L (z^\alpha,\hat \lambda)  d \bar Z(z, \alpha)^L\right\}(t)
\end{equation}

\end{prop}

\begin{proof} 

With the definitions of $N_K$ and $S_K$ given by  (\ref{rel-sol}) and (\ref{szego}), respectively, we shall show $\Box_b N_K=I-S_K$. On the transform side, we have
\begin{align*}
\Box_b^{v^\alpha, \hat \lambda} \left\{  N_K(z^\alpha , \hat \lambda) \right\} &=  
\sum_{L \in \I_q} C_{K,L}(\alpha) \Box_b^{v^\alpha, \hat \lambda} \left\{ \tilde N_L (z^\alpha, \hat \lambda) \,
d \bar Z(z, \alpha)^L \right\} \\
&=\sum_{L \in \I_q} C_{K,L}(\alpha) \Box_{LL}^{\hat \lambda} \big\{  \tilde N_L (z^\alpha, \hat \lambda)\big\} \,
d \bar Z(z, \alpha)^L  \\
&= (2 \pi)^{-m/2} \sum_{L \in \I_q} C_{K, L} (\alpha) \big[\delta_0(z^\alpha)- \tilde 
S_L(z^\alpha,\hat\lambda)\big] \, d \bar Z(z, \alpha)^L \ \ \ \textrm{from (\ref{NLsolve})} \\
&=  (2 \pi)^{-m/2}\delta_0(z)  \otimes 1_\lambda d \bar z^K -
(2 \pi)^{-m/2} \sum_{L \in \I_q} C_{K, L} (\alpha)\tilde S_L(z^\alpha,\hat\lambda) \, d \bar Z(z, \alpha)^L 
\ \ \textrm{from (\ref{C(K,L)})}
\end{align*}
where the function $1_\lambda$ is the constant function which is 1 in the $\lambda $ coordinates. 
Now take the inverse Fourier transform  (in $\lambda$) of both sides. The left side becomes $\Box_b N_K$ and 
then use the fact that $F_\lambda^{-1} \{ (2 \pi )^{-m/2} 1_\lambda \}(t) = \delta_0(t)$ and we
obtain
\begin{align*}
\Box_b \{N_K(z,t) \} &= \delta_0(z) \delta_0 (t) \, d \bar z^K  -
\opF_{\lambda}^{-1}\bigg\{(2 \pi)^{-m/2}\sum_{L \in \I_q} C_{K, L} (\alpha)\tilde S_L(z^\alpha,\hat\lambda) \, d \bar Z(z, \alpha)^L\bigg\} \\
&= \delta_0(z) \delta_0 (t) \, d \bar z^K  - S_K(z,t) \ \ \ \textrm{using } (\ref{szego})
\end{align*}
as desired. 
\end{proof}

%
%
\section{A New Derivation of the Integral Formula -- Proof of Theorem \ref{thm:new derivation}}\label{sec:new derivation}

\begin{proof}[Proof of Theorem \ref{thm:new derivation}]
We first assume the Szeg\"o kernel is zero, that is, $|\Gamma_L|=\emptyset$ for all $L\in\I_q$. 
Consequently, it follows from (\ref{Nfromheat}) that
$\tilde N_L(z^\alpha,\hat \lambda) = \int_0^\infty \tilde H_L(s, z^\alpha,\hat \lambda) \, ds$.
To prepare for the calculation of $\tilde N_L(z, t)$, we use
polar coordinates and write $\lambda = \alpha \tau$  where $\alpha $ belongs to the unit sphere $S^{m-1}$ and
$\tau >0$. We observe
\[
 \frac{2e^{s \eps_{j,L}^\alpha| \mu_j^\lambda|}  | \mu_j^\lambda|}
{\sinh (s |\mu_j^\lambda| )} = \frac{4e^{s \tau \eps_{j, L}^\alpha |\mu_j^\alpha| }
|\mu_j^\alpha | \tau}{e^{s| \mu_j^\alpha| \tau}
- e^{-s |\mu_j^\alpha| \tau}} =
\frac{4e^{s \tau( \eps_{j,L}^\alpha -1) |\mu_j^\alpha| }|\mu_j^\alpha | \tau}{1
- e^{- 2s |\mu_j^\alpha| \tau}}
\]
and
\[
 \coth (s |\mu_j^\lambda|) = \frac{e^{s \tau |\mu_j^\alpha|} + e^{-s \tau |\mu_j^\alpha |}}
{e^{s \tau |\mu_j^\alpha| } - e^{-s \tau |\mu_j^\alpha |}}
= \frac{1 + e^{-2s \tau |\mu_j^\alpha |}}
{1 - e^{-2s \tau |\mu_j^\alpha |}}.
\]

We now recover  $ N_K (z,t) $ from (\ref{rel-sol}) by computing the inverse Fourier transform using  polar coordinates ($\lambda = \alpha \tau, \ \alpha \in S^{m-1}, \  \tau>0$).
\begin{align*}
N_K ( z,t) &= (2 \pi)^{-m/2} \sum_{L \in \I_q} \int_{s=0}^\infty  \int_{\lambda \in \R^m} C_{K,L}(\alpha)  d \bar Z(z, \alpha)^L \tilde H_L (s, z, \hat \lambda) e^{it \cdot \lambda} \, d \lambda \, ds \\
&= (2 \pi)^{-m/2}  \sum_{L \in \I_q}\int_{s=0}^\infty \int_{\tau =0}^\infty \int_{\alpha \in S^{m-1}}  C_{K,L}(\alpha)  d \bar Z(z, \alpha)^L \tilde H_L (s, z, \widehat{\alpha \tau}) 
e^{i t \cdot \alpha \tau} \tau^{m-1} \, d \alpha\, d \tau \, ds
\end{align*}
where $d\alpha$ is the surface volume form on the unit sphere in $\R^m$.
\m

Let $r=e^{-2s \tau}$ in the $s$-integral and so $ds=-dr/(2 \tau r)$ and the oriented $r$-limits
of integration become $1$ to $0$.  We obtain
\begin{align*}
 N_K (z,t) &= \frac{4^n}{2(2 \pi)^{m+n}}  \sum_{L \in \I_q} \int_{r=0}^1  \int_{\alpha \in S^{m-1}} C_{K,L}(\alpha)  d \bar Z(z, \alpha)^L \\
& \int_{\tau =0}^\infty
\frac{1}{|\log r|^{n-\nu(\alpha)}} 
\prod_{j=1}^{\nu(\alpha)} \frac{r^{\frac 12(1-\eps_{j,L}^\alpha) |\mu_j^\alpha |}|\mu_j^\alpha |}{1- r^{|\mu_j^\alpha|}}
e^{-\tau(A_\alpha( r,z)-it\cdot\alpha)}  \tau^{n+m-2}\,  d \tau \, d \alpha \, \frac{dr}{r}
\end{align*}
where
\[
A_\alpha(r,z) = \frac{2}{|\log r|}|{z''}^\alpha|^2 +  \sum_{j=1}^{\nu(\alpha)} |\mu_j^\alpha| \left( 
\frac{1+ r^{|\mu_j^\alpha|}}{1- r^{|\mu_j^\alpha|}} \right) |z_j^\alpha|^2 .
\]

We now perform the $\tau$-integral by using the following formula:
\[
\int_{\tau=0}^\infty \tau^p e^{-a \tau} \, d \tau = \frac{p!}{a^{p+1}} \ \ \ \textrm{for } 
\Rre a >0
\]
which concludes the proof for \eqref{eqn:N no Szego}.
\m

Repeating this argument for the Szeg\"o kernel using (\ref{szego}), we have
\begin{align*}
S_K(z,t) &= \frac{4^n}{(2\pi)^{m+n}} \sum_{L \in \I_q} \int_{\alpha\in\Gamma_L}
C_{K,L}(\alpha)  d \bar Z(z, \alpha)^L \int_0^\infty \Big(\prod_{j=1}^n|\mu_j^\alpha|\Big)
\tau^{n+m-1} e^{-\tau(\sum_{j=1}^n |\mu_j^\alpha||z_j^\alpha|^2 - i \alpha\cdot t)}\, d\tau\, d\alpha \\
&= \frac{4^n(n+m-1)!}{(2\pi)^{m+n}} \sum_{L \in \I_q} \int_{\alpha\in\Gamma_L} 
C_{K,L}(\alpha)  d \bar Z(z, \alpha)^L
\frac{\prod_{j=1}^n|\mu_j^\alpha|}{(\sum_{j=1}^n |\mu_j^\alpha||z_j^\alpha|^2 - i \alpha\cdot t)^{n+m}}\, d\alpha 
\end{align*}
which concludes the proof for \eqref{eqn:Szego}.
\m

Finally, if $S_L(z,t)\neq 0$, then using (\ref{rel-sol}) and (\ref{Nfromheat})
\begin{align*}
& N_K(z,t) \\
&= \frac{1}{(2\pi)^{\frac m2}} \int_{s=0}^\infty  \sum_{L \in \I_q} \left( \int_{\alpha\in\Gamma_L}
C_{K,L}(\alpha)  d \bar Z(z, \alpha)^L
\int_{\tau=0}^\infty 
\Big(\tilde H_L(s,z,\widehat{\tau\alpha}) - \tilde S_L(z,\widehat{\tau\alpha})\Big)e^{i\tau(\alpha\cdot t)} \tau^{m-1}\, d\tau\, d\alpha \right) \, ds  \\
&+  \frac{1}{(2\pi)^{\frac m2}} \int_{s=0}^\infty  \sum_{L \in \I_q} \left( \int_{\alpha\in\Gamma_L}
 \int_{\alpha\not\in\Gamma_L}
C_{K,L}(\alpha)  d \bar Z(z, \alpha)^L
\int_{\tau=0}^\infty 
\tilde H_L(s,z,\widehat{\tau\alpha})e^{i\tau(\alpha\cdot t)} \tau^{m-1}\, d\tau\, d\alpha
\right) \, ds \\
&= I_K +II_K.
\end{align*}
The second set of integrals is virtually identical to what we computed earlier and we get
\begin{align*}
&II_K = \\
& \frac{4^n(n+m-2)!}{2(2 \pi)^{m+n}}  \int_{r=0}^1  \sum_{L \in \I_q} \left( \int_{\alpha\not\in\Gamma_L}
C_{K,L}(\alpha)  d \bar Z(z, \alpha)^L
\prod_{j=1}^n \frac{r^{\frac 12(1-\eps_{j,L}^\alpha) |\mu_j^\alpha |}|\mu_j^\alpha |}{1- r^{|\mu_j^\alpha|}}
\frac{d \alpha }{  (A(r,z)-i \alpha \cdot t)^{n+m-1}} \right) \, \frac{dr}{r}
\end{align*}
where
\[
A_\alpha(r,z) =   \sum_{j=1}^n |\mu_j^\alpha| \left(\frac{1+ r^{|\mu_j^\alpha|}}{1- r^{|\mu_j^\alpha|}} \right) |z_j^\alpha|^2 .
\]
For the first set of integrals, we observe that 
\begin{align*}
I_K &=
\frac{4^n}{2(2\pi)^{m+n}}
\sum_{L \in \I_q} \left( \int_{r=0}^1 \int_{\alpha\in\Gamma_L} 
C_{K,L}(\alpha)  d \bar Z(z, \alpha)^L
\int_{\tau=0}^\infty \right. \\
& \left. \Big[\Big(\prod_{j=1}^n \frac{|\mu_j^\alpha|}{1-r^{|\mu_j^\alpha|}}\Big) e^{-\tau(A_\alpha(r,z)-it\cdot\alpha)} 
- \prod_{j=1}^n|\mu_j^\alpha|e^{\tau(A_\alpha(0,z)-it\cdot\alpha)}\Big] \tau^{n+m-2}\,d\tau\,d\alpha\,\frac{dr}r \right) \\
&= \frac{4^n(n+m-2)!}{2(2\pi)^{m+n}} \sum_{L \in \I_q} \left( 
\int_{r=0}^1 \int_{\alpha\in\Gamma_L}
C_{K,L}(\alpha)  d \bar Z(z, \alpha)^L \right. \\
& \left. \Big[\Big(\prod_{j=1}^n \frac{|\mu_j^\alpha|}{1-r^{|\mu_j^\alpha|}}\Big) \frac{1}{(A_\alpha(r,z)-it\cdot\alpha)^{n+m-1}} 
- \frac{\prod_{j=1}^n|\mu_j^\alpha|}{(A_\alpha(0,z)-it\cdot\alpha)^{n+m-1}}\Big] \,d\alpha\,\frac{dr}r \right).
\end{align*}
This completes the proof of (\ref{eqn:N with Szego}).

That the convergence of the resulting integrals is absolute follows from a straightforward Taylor expansion argument around $r=0$ and
$r=1$, the only possible points where the integrand appears to blow up.
\end{proof}

%
%
\section{Examples}\label{sec:examples}
We analyze three examples in this section, all of which fall into the cases discussed in Remark \ref{rem:C_KL diagonal}
so the formulas from Theorem \ref{thm:new derivation} are slightly simpler. We discuss
codimension $2$ quadrics in $\C^4$ when $q=0,2$, the Heisenberg group (so $m=1$), and the product the Heisenberg groups (so we fall into
the sum of squares case).

\subsection{Codimension $2$ quadrics in $\C^4$}
When $n=m=2$, we wrote down the formulas for $N$ in the case of three canonical examples \cite{BoRa13q}:

\begin{itemize}

\item $M_1$ where $\phi(z,z)=(|z_1|^2, |z_2|^2)^T$ 
\item $M_2$ where $\phi(z,z) = (2 \Rre (z_1 \z_2), |z_1|^2-|z_2|^2)^T$
\item $M_3$ where $\phi(z,z) = (2 |z_1|^2, 2 \Rre (z_1 \z_2))^T$

\end{itemize}

These examples are canonical in the sense that any quadric in $\C^2 \times \C^2$ whose Levi form has 
image which is {\em not} contained in a one-dimensional cone is biholomorphic to one of these three examples 
(see  \cite{Bog91}). Additionally, these three examples perfectly demonstrate the three possibilities for solvability/hypoellipticity of $\Boxb$ 
on quadrics.

The quadric $M_1$ is simply a Cartesian product of Heisenberg groups and both solvability and hypoellipticity are impossible for any degree. 
In this case, 
\[
A^\lambda = \begin{pmatrix} \lambda_1 & 0 \\ 0 & \lambda_2 \end{pmatrix},
\]
so the eigenvalues of $A^\lambda$ are $\lambda_1$ and $\lambda_2$, so $\{(n^+(\lambda),n^-(\lambda))\} = \{(2,0),(1,0),(1,1),(0,1),(0,2)\}$.

For $M_2$, it follows from Peloso and Ricci \cite{PeRi03} that solvability and hypoellipticity occur for $(0,q)$ forms if and only if 
 $q=0$ or $q=2$. In this case,
 \[
 A^\lambda = \begin{pmatrix} \lambda_2 & \lambda_1 \\ \lambda_1 & -\lambda_2 \end{pmatrix}
 \]
 which gives us eigenvalues $\pm |\lambda|$, so that for all $\lambda\in\R^2\setminus\{0,0\}$, $(n^+(\lambda),n^-(\lambda)) = (1,1)$.
 Additionally, we showed that
the complex Green operator is given by (group) convolution with respect to the kernel 
\[
N_2 (z,t) = C (|z|^4+|t|^2)^{-3/2}, 
\]
where $C$ is a constant \cite[Theorem 3]{BoRa13q}. 

For $M_3$, $\Boxb$ is solvable if and only if $q=0$ or $2$ and is never hypoelliptic. In this case,
\[
A^\lambda = \begin{pmatrix} 2\lambda_1 & \lambda_2 \\ \lambda_2 & 0 \end{pmatrix}
\]
so that the eigenvalues are $\lambda_1 \pm |\lambda|$. Thus $\{(n^+(\lambda),n^-(\lambda))\} = \{(1,0),(1,1),(0,1)\}$ with the degenerate
values occurring when $\lambda_2=0$.
In Corollary \ref{cor:NforM3} below, 
we give a more useful formula for $N$ on $M_3$ \cite{BoRa13q}. The
analysis of the operator is extremely complicated and delicate and is the subject of a later work in the series
\cite{BoRa21III}. We must mention the paper of Nagel, Ricci, and Stein which 
analyzes $L^p$ estimates on a class of higher codimension quadrics in $\C^n \times \C^m$ 
which depend only on $|z_j|^2$, $1 \leq j \leq m$ \cite{NaRiSt01}. However, their result 
applies to neither $M_2$ nor $M_3$ for these quadrics canot be described in this manner.

\subsection{Example $M_3$.} \label{subsec:M3formula} 
Let $q=0$. As defined in Section \ref{sec:intro},
\[
M_3= \{ (z,w) \in \C^2 \times \C^2: \ 
\Imm w_1 = 2|z_1|^2, \ \Imm w_2 = 2 \Rre (z_1 \bar z_2) \}.
\]
Here, $m=n=2$, and for $\alpha=(\cos \theta, \sin \theta)$, we easily compute
$\mu_1^\alpha = 1+\cos \theta$, $\mu_2^\alpha = \cos \theta -1$.
The function $\phi$ satisfies both $A_0$ and $A_2$, though we will concentrate on the case $q=0$ (the case for $q=2$ is similar).
Since $L=\{0\}$,  $\eps_j^\alpha = - \sgn(\mu_j^\alpha )$ and so
$\eps_1^\alpha = -1$ and $\eps_2^\alpha = +1$ (except when $\theta =0$ or $1$ which is  a
set of measure zero). 
We obtain
\begin{align}
\label{Nformula}
N(z,t) =(2 \pi)^{-4} \int_{\theta =0}^{2 \pi} \int_{r=0}^1 
& 4^2\frac{ r^{\cos \theta} \sigma_1 (\theta) \sigma_2 (\theta)}
{(1-r^{\sigma_1 (\theta)})(1-r^{\sigma_2 (\theta)})} \\
&\times \frac{dr \, d \theta}{ (-i \alpha (\theta) \cdot t +  \sigma_1 (\theta) E_1 (r, \theta) |z_1^\theta|^2
+  \sigma_2 (\theta)  E_2 (r, \theta) |z_2^\theta|^2)^3} \nn
\end{align}
where
\[
\alpha (\theta) = ( \cos \theta, \sin \theta), \ \ 
\sigma_1 (\theta) = 1+ \cos \theta, \ \ \sigma_2 (\theta) = 1- \cos \theta, \ \ 
E_j(r, \theta) = \frac{1+r^{\sigma_j (\theta)}}{ 1- r^{\sigma_j (\theta)}}.
\]
We first wrote this formula in \cite{BoRa13q}. We wish to express it in a more useful and computable form which we will
use in \cite{BoRa21III}.

We let $t=(t_1, t_2)$ which gives $\alpha (\theta) \cdot t =t_1 \cos\theta +t_2 \sin \theta$. We also let
\[
x=r^{\sigma_1}, \ \textrm{so} \ dx = \sigma_1 r^{\sigma_1-1} dr \ \textrm{and} \ 
\sigma = \frac{\sigma_2}{\sigma_1} = \frac{1- \cos \theta }{1+\cos \theta}, \ d \theta = \frac{d \sigma}{(\sigma+1) \sqrt{\sigma}}
\]
and obtain
\[
\cos \theta = \frac{1- \sigma}{1+ \sigma} \ \ \ \textrm{and } \ \ \sin \theta = \frac{\pm 2 \sqrt{\sigma}}{1+ \sigma}
\]
where $\pm$ is $+$ for $\theta \in [0, \pi]$ and $-$ for $\theta \in (\pi, 2 \pi]$. Also the interval $0 \leq \theta \leq \pi$
corresponds to the oriented $\sigma$ interval  $[0 , \infty)$ and the interval $\pi \leq \theta \leq 2\pi$ corresponds to 
$(\infty, 0]$. In Theorem \ref{thm:new derivation}, the point $z$ is expressed in terms of the eigenvectors of $\phi^\lambda$. To this end, we set
\begin{align*}
z_1\{\sqrt\sigma\} &= \frac{1}{\sqrt{1+\sigma}} \big(z_1 + \sqrt{\sigma} z_2 \big)  \\
z_2\{\sqrt\sigma\} &= - \frac{1}{\sqrt{1+\sigma}} \big(\sqrt \sigma z_1 -z_2 \big) 
\end{align*}
and
\begin{align*}
\tilde z_1\{\sqrt\sigma\} &= -z_1\{\sqrt\sigma\} = \frac{1}{\sqrt{1+\sigma}} \big(-z_1 + \sqrt{\sigma} z_2 \big)  \\
\tilde z_2\{\sqrt\sigma\} &= -z_2\{\sqrt\sigma\} = - \frac{1}{\sqrt{1+\sigma}} \big(\sqrt \sigma z_1 +z_2 \big) 
\end{align*}
We then obtain the following corollary to Theorem \ref{thm:new derivation}:

\begin{cor} \label{cor:NforM3}
The fundamental solution to $\Box_b$ for $M_3$ on functions is given by 
convolution with the kernel
\begin{align}
&N (z,t)\nn\\
& = 2{(2 \pi)^{-4} }  \int_{\sigma =0}^{ \infty} \int_{x=0}^1  
\frac{ \sqrt{\sigma} (\sigma+1)  }{ (1-x)(1-x^\sigma )}  \frac{ dx \, d \sigma}{\left[ - i \left( t_1 \frac{1-\sigma }{2} + t_2 \sqrt{\sigma} \right) + \left(\frac{1+x}{1-x} \right) |z_1\{\sqrt\sigma\}|^2
+ \sigma  \left( \frac{1+x^\sigma}{1-x^ \sigma} \right) |z_2\{\sqrt\sigma\}|^2 \right]^3}\nn \\
&+ 2(2 \pi)^{-4}  \int_{\sigma =0}^{ \infty} \int_{x=0}^1  
\frac{ \sqrt{\sigma} (\sigma+1)  }{ (1-x)(1-x^\sigma )}  \frac{ dx \, d \sigma}{\left[ - i \left( t_1 \frac{1-\sigma }{2} - t_2 \sqrt{\sigma} \right) + \left(\frac{1+x}{1-x} \right) |\tilde z_1\{\sqrt\sigma\}|^2
+ \sigma  \left( \frac{1+x^\sigma}{1-x^ \sigma} \right) |\tilde z_2\{\sqrt\sigma\}|^2 \right]^3}.\nn
\end{align}
\end{cor}
This formula is the launching point for \cite{BoRa21III}.

\subsection{The Heisenberg group} Denote the Heisenberg group $\opH^n \cong \R^{2n}\times\R$.
The Kohn Laplacian $\Boxb$ has a nontrivial kernel in the case that $L = \emptyset$ or $L = \{1,\dots,n\}$. The calculation for these
two cases is identical and we prove the details in the case $L = \{1,\dots,n\}$.
A derivation of a related formula from the classical methods appears in \cite[pp.615-617]{Ste93}.
We set
\begin{equation}\label{eqn:log a}
\log\Big(\frac{|z|^2-it}{|z|^2+it}\Big) = \log(|z|^2-it) - \log(|z|^2+it)
\end{equation}
for all $z\in\C^n$ and $t\in\R$ and assume that the logarithm is defined via the principal branch.
\begin{thm}\label{eqn:relative inverse on the H group} On the Heisenberg group $\opH^n$, 
\begin{enumerate}[1.]
\item The relative fundamental solution $e^{-s\Boxb}(I-S_0)$ to $\Boxb = \dbarbs\dbarb$ on functions is given by the integration kernel
\[
N_{\emptyset}(z,t) = \frac{2^{n-2}(n-1)!}{\pi^{n+1}} \frac{1}{(|z|^2+it)^n}\Big[ \log\Big(\frac{|z|^2+it}{|z|^2-it}\Big)- \sum_{j=1}^{n-1}\frac 1j\Big].
\]
\item The relative fundamental solution $e^{-s\Boxb}(I-S_n)$ to $\Boxb = \dbarb\dbarbs$ on $(0,n)$-forms is given by the integration kernel
\[
N_{\{1,\dots,n\}}(z,t) = \frac{2^{n-2}(n-1)!}{\pi^{n+1}} \frac{1}{(|z|^2-it)^n}\Big[ \log\Big(\frac{|z|^2-it}{|z|^2+it}\Big)- \sum_{j=1}^{n-1}\frac 1j\Big].
\]
\end{enumerate}
\end{thm}
\begin{rem}
\begin{enumerate}
\item Up to a function in $\ker\Boxb$,  our formula appears to be the complex conjugate of the formula in \cite[Chapter XIII, Equation (51)]{Ste93}.
This is a consequence of the fact that our computations are taken with respect to
right invariant vector fields and not left invariant vector fields.
\item For a discussion regarding the consequences of the existence of a relative fundamental solution, we again refer the reader to
\cite[Chapter XIII, Section 4.2]{Ste93}. It is easy to see that the convolution $N$ with a Schwartz function will be an object in $L^2$ and hence
orthogonal to $\ker\Boxb$.
\end{enumerate}
\end{rem}

\begin{proof}
Since $L = \{1,\dots,n\}$, the Szeg\"o kernel $S(z,\hat\lambda) = S_L(z,\hat\lambda)$ has support
$\supp S_L(z,\hat\lambda) =[0,\infty)$ which means (suppressing $L$)
\[
N(z,t) = \frac{1}{\sqrt{2\pi}} \int_0^\infty \int_0^\infty \big(H(s,z,\hat\lambda)-S(z,\hat\lambda)\big)e^{it\lambda}\, ds\, d\lambda
+  \frac{1}{\sqrt{2\pi}} \int_{-\infty}^0 \int_0^\infty H(s,z,\hat\lambda) e^{it\lambda}\, ds\, d\lambda
\]
Equation (\ref{eqn:N with Szego}) yields
\begin{equation}\label{eqn:integral for N}
N(z,t)= \frac{4^n(n-1)!}{2(2\pi)^{n+1}}\int_0^1 \frac 1r\Big[\frac1{(1-r)^n} \frac{1}{(\frac{1+r}{1-r}|z|^2-it)^n} - \frac{1}{(|z|^2-it)^n}\Big]
+  \frac{r^{n-1}}{(1-r)^n} \frac{1}{(\frac{1+r}{1-r}|z|^2+it)^n}\, dr
\end{equation}
Set $a = \frac{|z|^2-it}{|z|^2+it}$ and for $\delta>0$, $a_{\delta} =  \frac{|z|^2+\delta-it}{|z|^2+\delta+it}$ (so $|a|=|a_{\delta}|=1$). The reason that we
introduce $a_{\delta}$ is that a logarithm appears in the integral, and $\log a$ is not well defined with the principal branch if $|z|^2=0$. By introducing
$\delta$, it is immediate that for any $a_{\delta}$
\[
\log a_{\delta} = \log(|z|^2+\delta-it) - \log(|z|^2+\delta+it)
\]
and by sending $\delta\to 0$, we obtain $\log a$ as in \eqref{eqn:log a}.
Ignoring the constants, we compute
\begin{align*}
I_\delta &= \int_0^1 \frac 1r \Big[\frac{1}{((1+r)(|z|^2+\delta)-it(1-r))^n} - \frac{1}{(|z|^2+\delta-it)^n}\Big] +
 \frac{r^{n-1}}{((1+r)(|z|^2+\delta)+it(1-r))^n}\, dr \\
&= \int_0^1 \frac 1r\Big[\frac{1}{((|z|^2+\delta+it)r+|z|^2+\delta-it)^n} - \frac{1}{(|z|^2+\delta-it)^n}\Big]+ 
\frac{r^{n-1}}{((|z|^2+\delta-it)r + (|z|^2+\delta+it))^n}\, dr \\
&= \frac{1}{(|z|^2+\delta+it)^n} \bigg[\int_0^1\Big( \frac{1}{(r+a_{\delta})^n} -\frac1{a_{\delta}^n}\Big)\, \frac{dr}r 
+ \int_0^1 \frac{r^{n-1}}{(a_{\delta} r + 1)^n}\, dr \bigg]
\end{align*}
For the second integral, we change variables $r = \frac1s$ and compute
\[
\int_0^1 \frac{s^{n-1}}{(a_{\delta} s+1)^n}\, ds = \int_1^\infty \frac{1}{(r+a_{\delta} )^n}\frac{dr}r.
\]
Thus,
\[
(|z|^2+\delta+it)^n I_\delta = \lim_{\ep\to 0} \bigg[\int_\ep^\infty \frac{1}{(r+a_{\delta} )^n}\frac{dr}r +\frac{1}{a_{\delta} ^n}\log\ep\bigg].
\]
A geometric series argument shows that
\[
\frac{1}{r(r+a_{\delta} )^n} = \frac{1}{a_{\delta} ^n r} - \frac{1}{a_{\delta} ^n(r+a_{\delta} )} - \sum_{j=1}^{n-1} \frac{1}{a_{\delta} ^{n-j}(r+a_{\delta} )^{j+1}}.
\]
Therefore
\begin{align*}
(|z|^2+\delta+it)^n I_\delta 
&= \lim_{\ep\to 0} \bigg[\int_\ep^\infty  \frac{1}{a_{\delta} ^n r} - \frac{1}{a_{\delta} ^n(r+a_{\delta} )}\, dr -\sum_{j=1}^{n-1} \int_\ep^\infty\frac{1}{a_{\delta} ^{n-j}(r+a_{\delta} )^{j+1}}\, dr
 +\frac{1}{a_{\delta} ^n}\log\ep\bigg] \\
&= \lim_{\ep\to 0}\bigg[\frac{\log(a_{\delta} +\ep)}{a_{\delta} ^n} -\sum_{j=1}^{n-1}\frac{1}{ja_{\delta} ^{n-j}(a_{\delta} +\ep)^j}\bigg]
= \frac{1}{a_{\delta} ^n}\Big(\log a_{\delta}  - \sum_{j=1}^{n-1}\frac 1j\Big).
\end{align*}
Thus, if we set $N_\delta(z,t)$ to equal the right hand side of \eqref{eqn:integral for N} except with $|z|^2$ replaced by $|z|^2+\delta$,
then $\frac{|z|^2+\delta -it}{|z|^2+\delta -it}$ stays away from the branch cut and
\begin{align*}
N_\delta(z,t) 
&= \frac{2^{n-2}(n-1)!}{\pi^{n+1}} \frac{1}{(|z|^2+\delta-it)^n}\Big[ \log\Big(\frac{|z|^2+\delta-it}{|z|^2+\delta+it}\Big)- \sum_{j=1}^{n-1}\frac 1j\Big]\\
&= \frac{2^{n-2}(n-1)!}{\pi^{n+1}} \frac{1}{(|z|^2+\delta-it)^n}\Big[ \log(|z|^2+\delta-it) - \log(|z|^2+\delta+it) - \sum_{j=1}^{n-1}\frac 1j\Big].
\end{align*}
This function is continuous in $\delta$, thus we may send $\delta\to 0$ and obtain the theorem.
\end{proof}

\subsection{The Cartesian product of Heisenberg groups} In contrast to the explicit computability of the 
Heisenberg group case, if
\[
M = \{(z,w)\in \C^2\times\C^2 : \Imm w_j = |z_j|^2\},
\]
$L = \{1,2\}$, and $\alpha = (\cos\theta,\sin\theta)$, then $\Gamma^\alpha_{\{1,2\}}$ is the first quadrant and
from Theorem \ref{thm:new derivation}, we have 
\begin{align*}
&N_{\{1,2\}}(z,t) =\frac{1}{\pi^4}
\int_{r=0}^1 \int_{\frac\pi 2}^\pi |\cos\theta\sin\theta|
\frac{r^{|\cos\theta|}}{(1-r^{|\cos\theta|})(1-r^{|\sin\theta|})} \frac{1}{(A_{\alpha}(r)-i(t_1\cos\theta+t_2\sin\theta))^{n+m-1}} \,d\theta\,\frac{dr}r \\
&+\frac{1}{\pi^4}\int_{r=0}^1 \int_\pi^{\frac{3\pi} 2} |\cos\theta\sin\theta|
\frac{r^{|\cos\theta+\sin\theta|}}{(1-r^{|\cos\theta|})(1-r^{|\sin\theta|})} \frac{1}{(A_{\alpha}(r)-i(t_1\cos\theta+t_2\sin\theta))^{n+m-1}}  \,d\theta\,\frac{dr}r\\
&+\frac{1}{\pi^4}\int_{r=0}^1 \int_{\frac{3\pi}2}^{2\pi} |\cos\theta\sin\theta|
\frac{r^{|\sin\theta|}}{(1-r^{|\cos\theta|})(1-r^{|\sin\theta|})} \frac{1}{(A_{\alpha}(r)-i(t_1\cos\theta+t_2\sin\theta))^{n+m-1}} \,d\theta\,\frac{dr}r\\
&+\frac{1}{\pi^4} \int_{r=0}^1 \int_0^{\frac\pi 2} \cos\theta\sin\theta
\Big[\frac{1}{(1-r^{\cos\theta})(1-r^{\sin\theta})} 
\frac{1}{(A_{\alpha}(r)-i(t_1\cos\theta+t_2\sin\theta))^{n+m-1}} \\
&\hspace{3in}- \frac{1}{(A_{\alpha}(0)-i(t_1\cos\theta+t_2\sin\theta))^{n+m-1}}\Big] \,d\theta\,\frac{dr}r 
\end{align*}
where
\[
A_{\alpha}(r) =  |\cos\theta| \left(\frac{1+ r^{|\cos\theta|}}{1- r^{|\cos\theta|}} \right) |z_1|^2 
+  |\sin\theta| \left(\frac{1+ r^{|\sin\theta|}}{1- r^{|\sin\theta|}} \right) |z_2|^2.
\]
On the other hand, using (\ref{eqn:Szego}), we compute the Szeg\"o kernel
\begin{align*}
S_{\emptyset}(z,t)
= \frac{6}{\pi^4}\int_\pi^{\frac{3\pi}2} \frac{\cos\theta\sin\theta}{((|z_1|^2+it_1)\cos\theta+(|z_2|^2+it_2)\sin\theta)^4}\, d\theta
= \frac{1}{\pi^4(|z_1|^2+it_1)^2(|z_2|^2+it_2)^2}.
\end{align*}

\bibliographystyle{alpha}
\bibliography{mybib}

\end{document}